\documentclass[10pt]{amsart}

\usepackage{amsmath, amsfonts, amssymb, xypic}
\usepackage{latexsym, verbatim}
\usepackage{graphicx}
\usepackage{color}
\usepackage{url}
\usepackage[table]{xcolor}
\usepackage{hyperref}
\usepackage{amsmath}

\usepackage{enumitem}

\newtheorem{thm}{Theorem}[section]
 \newtheorem{cor}[thm]{Corollary}
 \newtheorem{lem}[thm]{Lemma}
 \newtheorem{prop}[thm]{Proposition}

\theoremstyle{definition}
 \newtheorem{ex}[thm]{Example}
 \newtheorem{rmk}[thm]{Remark}

\theoremstyle{theorem}

\newcommand{\Z}{\mathbb{Z}}
\newcommand{\R}{\mathbb{R}}
\newcommand{\CC}{\mathbb{C}}

\newcommand{\PP}{\mathbb{P}}

\newcommand{\cA}{{\mathcal A}}

\newcommand{\Hh}{{\mathcal H}}

\newcommand{\ot}{\otimes}

\newcommand{\Ke}{{\rm Ker}\,}

\newcommand{\del}{\partial}
\newcommand{\delb}{{\bar \partial}}
\newcommand{\mub}{{\bar \mu}}

\newcommand{\Img}{\mathrm{Im}}
\newcommand{\Dol}{\mathrm{Dol}}
\newcommand{\dR}{\mathrm{dR}}

\title[Topological and geometric aspects of almost K\"{a}hler manifolds]
{Topological and geometric aspects of almost K\"{a}hler manifolds via harmonic theory}

\author[J. Cirici]{Joana Cirici}
\address[J. Cirici]{Department of Mathematics and Computer Science, Universitat de Barcelona\\
Gran Via 585\\
08008 Barcelona }
\email{jcirici@ub.edu}

\author[S. Wilson]{Scott O. Wilson}
  \address[S.~Wilson]{Department of Mathematics, Queens College, City University of New York, 65-30 Kissena Blvd., Flushing, NY 11367}
  \email{scott.wilson@qc.cuny.edu}

\thanks{J. Cirici would like to acknowledge partial support from the AEI/FEDER, UE (MTM2016-76453-C2-2-P) and the Serra H\'{u}nter Program.
}

\thanks{S. Wilson acknowledges support provided by a PSC-CUNY Award, jointly funded by The Professional Staff Congress and The City University of New York.}


\begin{document}

\begin{abstract}  The well-known K\"{a}hler identities naturally extend to the non-integrable setting.  This paper deduces several geometric and topological consequences of these extended identities for compact almost K\"{a}hler manifolds.
 Among these are identities of various Laplacians, generalized Hodge and Serre dualities, a generalized hard Lefschetz duality, and a Lefschetz decomposition, all on the space of $d$-harmonic forms of pure bidegree.  There is also a generalization of Hodge Index Theorem for compact almost K\"ahler $4$-manifolds. In particular, these provide topological bounds on the dimension of the space of $d$-harmonic forms of pure bidegree, as well as several new obstructions to the existence of a symplectic form compatible with a given almost complex structure.

\end{abstract}

\maketitle


\section{Introduction}

K\"{a}hler manifolds play a central role at the intersection of complex, symplectic and Riemannian geometry. Their striking set of properties arise most primitively from a set of purely local commutation relations, known as the \textit{K\"{a}hler identities}. In the compact setting, the theory of elliptic operators allows one to transfer these local statements into a set of surprising cohomological properties, called the \textit{K\"{a}hler package}.

In this paper we consider a set of \textit{almost K\"{a}hler identities} and establish a corresponding \textit{almost K\"{a}hler package} on the spaces of harmonic forms of pure bidegree. Topological and geometric implications are deduced in the compact case, since such harmonic forms include into the de Rham cohomology. The package includes relations among various Laplacians, generalized Hodge and Serre dualities, a generalized Hodge Theorem and hard Lefschetz duality, as well as a generalized Hodge Index Theorem for almost K\"ahler $4$-manifolds. Additional results are described in what follows.

Recall that the exterior differential of an almost complex manifold has four components, $d = \mub + \delb + \del + \mu$,
where $\mub$ and $\mu$ arise from the Nijenhuis tensor and vanish if and only if the structure is integrable.
In the presence of an almost Hermitian metric, there are formal adjoints for each component, and associated Laplacians.
Also, the fundamental $(1,1)$-form defines a Lefschetz operator $L$.
This leads to a natural set of almost K\"{a}hler identities involving the operators 
$\mub$, $\mu$, $\delb$, $\del$, $L$ and their adjoints. Such identities
have been previously noted by de Barolomeis and Tomassini in \cite{BaTo} (see also 
for instance \cite{Donaldson}, \cite{KotschickBourbaki} for the case of $\delb$, $\del$, $L$ and their adjoints). We remark that similar looking (but very different) identities are described by Verbitsky in \cite{Verbitsky} in the nearly K\"{a}hler setting, though these do not apply to the present context. 

We show that the almost K\"{a}hler identities imply, among other Laplacian relations, that 
\[\Delta_\delb+\Delta_\mu=\Delta_\del+\Delta_\mub.\]
We also show that the Laplacian with respect to the exterior differential satisfies
\[\Delta_d=2(\Delta_\delb+\Delta_\mu)+\text{mixed-bidegree terms}.\]
In the integrable case, the $\mub$- and $\mu$-Laplacians as well as the mixed bidegree terms vanish and one recovers the well-known K\"ahler identities $\Delta_\delb=\Delta_\del$
and $\Delta_d=2\Delta_\delb$.

In the compact case, the theory of harmonic forms allows one to translate the above local results into geometric and topological statements for almost K\"{a}hler manifolds, as we next explain.
Let $\delta$ denote one of the components $\mub, \delb, \del, \mu$. Define the space of \textit{$\delta$-harmonic forms} in bidegree $(p,q)$ by letting 
\[\Hh_\delta^{p,q}:=\Ke(\Delta_\delta)\cap \cA^{p,q},\]
where $\cA^{p,q}$ denotes the space of $(p,q)$-forms.
A priori, these spaces depend on the metric, which is indicated by the use of $\mathcal{H}$, conventional in Hodge theory.
Note that for compact almost Hermitian manifolds, the spaces $\Hh_\delb^{p,q}$ and $\Hh_\del^{p,q}$ are finite-dimensional by elliptic operator theory but in general, the spaces $\Hh_\mu^{p,q}$ and $\Hh_\mub^{p,q}$ are infinite-dimensional whenever they are non-zero.

We will consider the spaces of \textit{$\delb$-$\mu$-harmonic forms} given by the intersections
\[\Hh_\delb^{p,q}\cap \Hh_\mu^{p,q}.\]
These are identified with the kernel of the self-adjoint elliptic operator 
given by
$\Delta_\delb+\Delta_\mu$.
We will denote by
\[\ell^{p,q}:=\dim \left(\Hh_\delb^{p,q}\cap \Hh_\mu^{p,q}\right)=
\Ke(\Delta_\delb+\Delta_\mu)\cap\cA^{p,q}\]
the dimensions of these spaces.
In the compact integrable case, these are just the Hodge numbers of the manifold.

\newtheorem*{J0}{\normalfont\bfseries Theorem \textbf{\ref{AKharmonicequal}}}
\begin{J0}
 For any compact almost K\"ahler manifold of dimension $2m$ 
 and for all $(p,q)$, we have identities
 \[\Hh_d^{p,q}=\Hh_\delb^{p,q}\cap \Hh_\mu^{p,q}=\Hh_\del^{p,q}\cap\Hh_\mub^{p,q}.\]
 Also, the following dualities hold:
 \begin{enumerate}
  \item \emph{(Complex conjugation)}. We have equalities
  \[\Hh_\delb^{p,q}  \cap \Hh_\mu^{p,q}  =  \Hh_\delb^{q,p}  \cap \Hh_\mu^{q,p}.\]
\item \emph{(Hodge duality)}. The Hodge $\star$-operator induces isomorphisms
\[
\star : \Hh_\delb^{p,q}  \cap \Hh_\mu^{p,q} \to \Hh_\delb^{m-q,m-p}  \cap \Hh_\mu^{m-q,m-p}.
\]
\item \emph{(Serre duality)}. There are isomorphisms
\[\Hh_\delb^{p,q}  \cap \Hh_\mu^{p,q} \cong  \Hh_\delb^{m-p,m-q}  \cap \Hh_\mu^{m-p,m-q}.\]
 \end{enumerate}
 \end{J0}
 
The symmetries of the Hodge diamond for compact K\"{a}hler manifolds generalize here, giving 
\[\ell^{p,q}=\ell^{q,p}=\ell^{m-q,m-p}=\ell^{m-p,m-q}.\]
For almost K\"{a}hler manifolds, the spaces of $\delb$-$\mu$-harmonic forms are, in general, strictly contained in the spaces of $d$-harmonic forms of total degree $p+q$. In particular, in the compact case, we obtain inequalities
\[\sum_{p+q=k}\ell^{p,q}\leq b^k,\]
where $b^k$ denote the complex $k$-Betti number of the manifold.

Let us remark that the numbers $\ell^{*,*}$ are often non-trivial. For example, for any left-invariant almost K\"{a}hler structure on the Kodaira-Thurston manifold (see Subsection \ref{KT}), the numbers $\ell^{*,*}$ are given by  
\[
\xymatrix@C=.1pc@R=.1pc{
 &&1\\
 &1&&1\\
 0&&3&&0\\
 &1&&1\\
 &&1
}
\] 
Using Theorem \ref{AKharmonicequal} together with the various dualities of the spaces $\Hh_d^{p,q}$ we obtain better topological bounds:

\newtheorem*{J2}{\normalfont\bfseries Theorem \textbf{\ref{AKbounds}}}
\begin{J2}
For any compact almost K\"ahler manifold of dimension $2m$, the following is satisfied:
\begin{enumerate}
\item In odd degrees, we have 
\[\sum_{p+q=2k+1}\ell^{p,q}=2\sum_{0\leq p\leq k}\ell^{p,2k+1-p}\leq b^{2k+1}.\]
\item In even degrees, we have \[\sum_{p+q=2k}\ell^{p,q}=2\sum_{0\leq p<k}\ell^{p,2k-p}+\ell^{k,k}\leq b^{2k},\]
with $\ell^{k,k}\geq 1$ for all $k\leq m$.
\end{enumerate}
\end{J2}

The numbers $\ell^{p,q}$, and more generally the spaces 
$\Hh^{p,q}_d$, relate to various almost complex invariants studied in the literature. They inject into the cohomology spaces $H_J^{p,q}$ of Draghici, Li and Zhang. These consist of those cohomology classes which can be represented by a
complex closed form of type $(p,q)$ (see \cite{DrLiZh}). The spaces $\Hh^{p,q}_d$ also inject into the Dolbeault cohomology groups $H_{\Dol}^{p,q}$ for almost complex manifolds introduced by the authors in \cite{CWDol}. In fact, they inject into any page of the Fr\"{o}licher-type spectral sequence defined in 
\cite{CWDol} and in particular, into the de Rham page $H^{p,q}_{\dR}:=E_\infty^{p,q}$.
All these inclusions are, in general, strict.

The spaces $\Hh_d^{p,q}$ are also related to other invariants when taking the total degree. They inject into the spaces of symplectic harmonic forms $\Hh_{sym}^{p+q}$ introduced by Brylinski in \cite{Brylinski} and further exploited by Tseng and Yau in \cite{TY1,TY2}. In fact, there is an identity $\Hh_d^{p,q}=\Hh_{sym}^{p+q}\cap \cA^{p,q}$ when restricting to bidegree $(p,q)$-forms.
In the recent work \cite{TaTopreprint}, Tardini and Tomassini extend much of
the almost K\"{a}hler package developed here to the harmonic spaces 
$\Hh_{\delb+\mu}^{p+q}$, defined via the Laplacian of the operator $\delb+\mu$. 
These spaces properly contain the spaces $\Hh_d^{p,q}$ in the general almost K\"{a}hler case.

Theorem \ref{AKbounds} provides new obstructions for the existence of symplectic structures compatible with a given almost complex structure,
complementing the known topological-type obstructions arising from the symplectic form and the associated almost complex structure, as well as the results of Taubes via the theory of Seiberg-Witten invariants \cite{Taubes1}, \cite{Taubes2} (see also Gompf's review on symplectic obstruction theory \cite{Gompf}). 
For instance, for any almost K\"{a}hler structure on a manifold with the same cohomology as $\CC\PP^2$, its numbers $\ell^{*,*}$ coincide with the classical Hodge numbers of its canonical K\"{a}hler structure.

The new obstructions are most easily exploited by considering the following metric-independent spaces of \textit{holomorphic $p$-forms}, defined by 
\[\Omega^p_\delb:=\Ke(\delb)\cap \cA^{p,0}.\]
Theorem \ref{AKbounds} implies that for any almost complex structure on a manifold with $b^1\leq 1$,
a necessary condition for it to admit a compatible symplectic structure is that there are no holomorphic $1$-forms. 
This follows after showing that, for compact almost K\"{a}hler manifolds, the space of holomorphic $1$-forms coincides with the space of $(1,0)$-forms that are $d$-harmonic.

Holomorphic forms on almost K\"{a}hler manifolds are shown below to satisfy even further special properties.
Chen gave the first example of
how the Hodge numbers of a compact K\"{a}hler manifold affect its fundamental group  \cite{Chen}, \cite{Chen2}. We extend Chen's result to the non-integrable case, showing that if $M$ is a compact connected almost K\"{a}hler manifold satisfying 
\[\dim\Omega^1_\delb >\dim\Omega^2_\delb +1,\]
then $\pi_1(M)$ contains a free subgroup of rank $\geq 2$. In particular,
$\pi_1(M)$ is not solvable. This is an additional obstruction to finding a symplectic structure compatible with a given almost complex structure.

The almost K\"{a}hler identities also lead to a generalization of hard Lefschetz duality on $\Hh_d^{*,*}$. It is well known that symplectic manifolds generally do not satisfy hard Lefschetz duality in  cohomology. In fact, Mathieu points out that a necessary condition is that all odd Betti numbers must be even \cite{Mathieu}.
Nevertheless,  for compact almost K\"ahler manifolds we prove:

\newtheorem*{I3}{\normalfont\bfseries Theorem \textbf{\ref{HLef}}}
\begin{I3}[Generalized Hard Lefschetz Duality] 
For any compact almost K\"ahler manifold of dimension $2m$, the operators $\{L,\Lambda, H = [L,\Lambda] \}$ define a 
finite dimensional representation of $\mathfrak{sl}(2,\CC)$ on  
\[
\bigoplus_{p,q \geq 0} \Hh_d^{p,q} = \bigoplus_{p,q \geq 0} \Hh_\delb^{p,q} \cap \Hh_\mu^{p,q} = 
\Ke(\Delta_\delb+\Delta_\mu).
\]
Moreover, for all $0 \leq p \leq k\leq m$,
\[
L^{m-k} :  
\Hh_d^{p,k-p} 
\stackrel{\cong}{\longrightarrow} 
\Hh_d^{p+m-k,m-p}
\]
are isomorphisms. 
\end{I3}

Generalized Hard Lefschetz implies that the numbers $\ell^{p,q}$ satisfy
\[
\ell^{p,q} \leq \ell^{p+1,q+1} \leq \dots \leq \ell^{p+j,q+ j}
\]
for all $0 \leq p, q \leq m$ and $p + q + 2j \leq m$.
In particular, note that
if $b^k = 0$ for some $k \leq m$, then $ \ell^{p,q}=0$ for all $p+q \leq k$ with $p+q \equiv k \mod 2$.

In dimension $4$, the various symmetries yield only a few interesting numbers $\ell^{*,*}$ to consider, and the Betti number bounds involving $\ell^{1,1}$ can be improved using in the index of the intersection pairing:

\newtheorem*{I4}{\normalfont\bfseries Theorem \textbf{\ref{4mnfldIndex}}}
\begin{I4}
[Generalized Hodge Index Theorem, c.f. \cite{HZ}]
 For any compact $4$-dimensional almost K\"ahler manifold $M$, 
\[
\ell^{1,1}=  b_2^- + 1 \text{ and } b_2^+\geq 1,\]
 where the intersection pairing on $H^2(M;\CC)$ has index $(b_2^+,b_2^-)$. 
 \end{I4}

 A proof of the equality $\ell^{1,1}=  b_2^- + 1$ first appeared in the preprint \cite{HZ}, 
using Theorem \ref{AKharmonicequal} below, and following the establishment of an inequality, $\ell^{1,1} \leq b_2^- + 1$, first proved in a preprint of the present paper.

Using this result, we are able to conclude here that for any compact almost K\"{a}hler 4-manifold,  all of the numbers $\ell^{p,q}$  are metric-independent among all almost K\"{a}hler metrics that are compatible with the given almost complex structure.
This is related to a problem listed by Hirzebruch, in \cite{Hirzebruch}, and attributed to Kodaira-Spencer.
The problem includes two questions. The first asks whether the
dimensions of $\Hh_\delb^{p,q}$ are metric-independent on any almost complex manifold. This has been very recently answered negatively by Holt and Zhang in \cite{HZ}. The second part of the problem asks for 
a definition of metric-independent numbers generalizing Dolbeault cohomology to the non-integrable case. Note that while in the almost K\"{a}hler setting, the numbers $\ell^{p,q}$ are reasonable candidates, 
for general almost complex manifolds these numbers do not faithfully generalize Dolbeault cohomology.
Instead, the theory presented in \cite{CWDol} seems to exhibit more of the properties desirable for such a Dolbeault theory.

Finally, in the physics literature, the algebraic structure present on the differential forms of a K\"ahler manifold is often referred to as the $N=2$ supersymmetry algebra; see \cite{Zumino}, \cite{AF}, \cite{HKLR}. It has been noted that when the integrability condition is dropped, the supersymmetry is \emph{partially broken} \cite{FGR}.  The almost K\"{a}hler identities obtained here show that additional symmetries are indeed present, albeit in a more subtle and interesting way depending on the failure of integrability. We hope the results here may open new possibilities for physicists' construction and study of supersymmetric theories.

 This paper is organized as follows. In Section \ref{SecPrelim} we collect preliminaries on the differential forms on almost
 complex and Hermitian manifolds. In Section \ref{SecAKIds} we present the fundamental almost K\"{a}hler identities.
 In Section \ref{Section_consequences} we prove Theorems \ref{AKharmonicequal} and \ref{AKbounds}
 and detail several topological and geometric consequences.
 Section \ref{SecLef} is devoted to hard Lefshetz (Theorem \ref{Lefduality}) and the Hodge index for 4-manifolds (Theorem, \ref{4mnfldIndex}).
 Lastly, in Section \ref{SecNil}, we exhibit applications of the theory which show that the theoretical results allow us to perform new calculations, reaching new theoretical conclusions even in the case of nilmanifolds.
 In particular, the numbers $\ell^{*,*}$
 are computed using the topological bounds, and the failure of the symmetry 
 $ \Delta_\delb + \Delta_\mu =  \Delta_\del + \Delta_\mub$  is shown to detect almost complex structures that do not admit almost K\"ahler structures.

 \subsection*{Acknowledgments}
The authors thank Thomas Holt for pointing out an important sign error that appeared in a preprint of this paper.
We would also like to thank the anonymous referee for his suggestions.

\section{Differential forms on almost complex manifolds}\label{SecPrelim}

Let $(M,J)$ be an almost complex manifold and let
\[\cA^k:=\cA_{\dR}^k(M)\otimes_\R \CC=\bigoplus_{p+q=k}\cA^{p,q}\]
be the bigraded algebra of complex valued differential forms on $M$. The exterior differential decomposes as 
\[
d=  \mub+\delb+\del+\mu,
\]
with the components  $\mub$  and $\delb$ being complex conjugate to $\mu$ and $\del$, respectively. 
 Note that each component of $d$ is a derivation, with bidegrees given by 
\[
|\mub|=(-1,2), \,  |\delb|=(0,1), \, |\del|=(1,0), \text{ and } \,  |\mu|=(2,-1).
\]
In particular,  $\mub$ and $\mu$ are linear over functions.

One can show that $\mub + \mu$ is equal, up to a scalar, to the dual of the Nijenhuis tensor. In fact, 
\[
\mub + \mu = - \frac 1 4 \left( N_J \ot id_\mathbb{C} \right)^*,
\]
where the right hand side has been extended over all forms as a derivation. Since both sides are derivations, it suffices to check this on $1$-forms, which can be done using Cartan's formula relating the exterior differential and Lie bracket.
In particular, $J$ is integrable if and only if $\mub \equiv 0$.

Expanding the equation $d^2 =0$ we obtain the following set of equations:
\newcommand\numberthis{\addtocounter{equation}{1}\tag{$\vartriangle$}}
\begin{align*}
\mu^2 = 0 \\
\mu \del + \del \mu =  0 \\
\mu \delb + \delb \mu +\del^2 =  0 \\
\mu \mub + \del \delb + \delb \del + \mub \mu  = 0  \numberthis \label{eq;drelations}\\
\mub \del + \del \mub +\delb^2 =  0 \\
\mub \delb + \delb \mub =  0 \\
\mub^2 = 0
\end{align*}

For any almost Hermitian manifold $(M, J, \langle \,\, , \,\, \rangle)$ 
of dimension $2m$ 
there is an associated \textit{Hodge-star operator} 
\[
\star: \cA^{p,q} \to \cA^{m-q,m-p} \quad \quad \textrm{defined by} \quad \quad  
\alpha \wedge \star \bar \beta = \langle \alpha , \beta \rangle \textrm{vol}
\] 
where $\textrm{vol}$ is the volume form determined by the metric. 

There is an associated \emph{fundamental $(1,1)$-form} defined by
\[
\omega(X,Y): =  \langle J X,Y \rangle
\]
and \textit{Lefschetz operator}  
\[ L : \cA^{p,q} \longrightarrow \cA^{p+1,q+1},\text{ defined by }L(\eta) := \omega \wedge \eta.\]
It has adjoint $\Lambda = L^* = \star^{-1} L \star$. It  is well known that $\{ L, \Lambda, H=[L,\Lambda]\} $ defines a representation of $\mathfrak{sl}(2,\CC)$, with Lefschetz decomposition
on complex $k$-forms
\[
\cA^k = \bigoplus_{i \geq 0} L^i  P^{k-2i},
\]
where $P^j = \Ke( \Lambda ) \cap \cA^j$. The map  
\[L^{m-k}: P^k \longrightarrow \cA^{2m-k}\]
is injective for $k \leq m$ (see for instance \cite{Weil}).

The operators $\delta =  \mub, \delb, \del, \mu,$ and $d$ have $L_2$-adjoint operators $\delta^*$, and when $M$ is closed, one may check that 
\[
\mub^* = - \star \mu \star  \quad \text{and} \quad  \delb^* = - \star \del \star.
\]
The latter equation is well known. The first equation is checked similarly, by using the definition of $\star$ and the fact that $\mu$ is a derivation.

Define the \textit{$\delta$-Laplacian} by letting
\[
\Delta_\delta:= \delta \delta^*+\delta^*\delta.
\]
It satisfies 
\[
\star \Delta_{\bar \delta} = \Delta_{ \delta} \star .
\]
For all $p,q$, we will denote by
\[\Hh_\delta^{p,q}:= \Ke(\Delta_\delta) \cap \cA^{p,q} =  \Ke(\delta)\cap \Ke(\delta^*) \cap \cA^{p,q}\]
the space of \textit{$\delta$-harmonic forms} in bidegree $(p,q)$.

Note that $\Hh_\mu^{p,q}$ and $\Hh_\mub^{p,q}$  are infinite dimensional whenever they are non-zero, since $\mu$ and $\mub$ are linear over functions, 
but $\Hh_\delb^{p,q}$ and $\Hh_\del$ are finite dimensional on a compact manifold, by elliptic theory. Likewise, the space
\[\Hh^{p,q}_d:=\Ke(\Delta_d)\cap\cA^{p,q}\]
is finite-dimensional and
there is an inclusion
\[\bigoplus_{p+q=n}\Hh_d^{p,q}\subseteq\Hh_d^{p+q}.\]
which in general is strict.
In the compact case,
this inclusion gives topological bounds
\[\sum_{p+q=k}\dim \Hh_d^{p,q}\leq b^n\text{ for all }k\geq 0.\]

In contrast with the above inclusion, 
note that if $\delta$ is any of the operators $\mub$, $\delb$, $\del$ or $\mu$, there is an 
orthogonal direct sum decomposition
 \[
 \Ke(\Delta_\delta) \cap \cA^{k} = \bigoplus_{p+q =k} \Hh_\delta^{p,q},
 \]
since now both $\delta$ and $\delta^*$ are of pure bidegree. Note that a priori, all the spaces $\Hh_\delta^{*,*}$ and $\Hh_d^{*,*}$ depend on the chosen metric.

\begin{rmk}\label{Dolrema}
In \cite{CWDol} we define Dolbeault cohomology groups
$H_{\Dol}^{p,q}(M)$ associated to any almost complex manifold $M$.
Moreover, there is a Fr\"{o}licher-type spectral sequence $\{E_r^{*,*}(M)\}$ converging to complex de Rham cohomology and in particular, one obtains a bigrading on de Rham cohomology.
The spaces $\Hh_d^{p,q}$ are easily proven to inject in $E_r^{p,q}(M)$ for any $r\geq 1$. In particular, we always have inequalities 
\[\dim \Hh_d^{p,q}\leq h_{\dR}^{p,q}\text{ where }
h_{\dR}^{p,q}:=\dim E_\infty^{p,q}.\]
These inequalities refine the above topological bounds by almost complex invariants.
\end{rmk}

The following lemmas concerning almost Hermitian manifolds will be used later to deduce topological consequences for almost K\"ahler manifolds.

 \begin{lem} \label{dualities}
 For any compact almost Hermitian manifold of dimension $2m$
 and for all $0 \leq p,q \leq  m$ the following dualities hold:
 \begin{enumerate}
\item \emph{(Complex conjugation)}. We have equalities
\[\Hh_d^{p,q}   =  \Hh_d^{q,p}.\]
\item \emph{(Hodge duality)}. The Hodge $\star$-operator induces isomorphisms
\[
\star:\Hh_d^{p,q} \longrightarrow \Hh_d^{m-q,m-p}.
\]
\item \emph{(Serre duality)}. There are isomorphisms
\[\Hh_d^{p,q}  \cong  \Hh_d^{m-p,m-q}\text{ and }\Hh_\delta^{p,q} \to \Hh_\delta^{m-p,m-q}\text{ for }\delta=\mub,\delb,\del,\mu.\]
 \end{enumerate}
\end{lem}

\begin{proof}
The first duality follows from the identity
\[\Ke(\Delta_\delta)\cap \cA^{p,q}=\Ke(\Delta_{\bar\delta})\cap \cA^{q,p}\] 
for any $\delta=\mub,\delb,\del,\mu$.
Hodge duality follows from this same identity together with the relation
$\star \Delta_{\bar \delta} = \pm \Delta_{ \delta} \star$,
which also proves the Serre dualities.
\end{proof}

\begin{lem}\label{intersections}
 For any compact almost Hermitian manifold 
we have equalities 
\begin{enumerate}[itemsep=2mm]
\item $\Hh_d^{p,q}=\Ke(\Delta_\mub+\Delta_\delb+\Delta_\del+\Delta_\mu)\cap\cA^{p,q}=\Ke(\Delta_\mub+\Delta_\delb)\cap(\Delta_\del+\Delta_\mu)\cap\cA^{p,q}$,
\item $\Hh_\delb^{p,q}\cap \Hh_\mu^{p,q}=\Ke(\Delta_\delb+\Delta_\mu)\cap\cA^{p,q}$ and
\item $\Hh_\del^{p,q}\cap \Hh_\mub^{p,q}=\Ke(\Delta_\del+\Delta_\mub)\cap\cA^{p,q}$.
\end{enumerate}
\end{lem}
\begin{proof}
 If a form $\alpha$ of type $(p,q)$ satisfies $d\alpha=0$ and $d^*\alpha=0$,
then the four components of $d$ and $d^*$ also vanish by bidegree reasons.
Therefore, for any almost Hermitian manifold we have
\[
\Hh_\delb^{p,q} \cap \Hh_\mu^{p,q}  \cap \Hh_\del^{p,q} \cap \Hh_\mub^{p,q}=\Hh_d^{p,q} ,
\]
since the space on the left is in the eightfold intersection of kernels of $\mub, \delb, \del, \mu$ and their adjoints.

Let $S$ be a subset of $\{\mub,\delb,\del,\mu\}$, then
\[\alpha \in \Ke\left(\sum_{\delta\in S}\Delta_\delta\right)\]
if and only if 
\[0=\sum_{\delta\in S}\langle \Delta_\delta \alpha,\alpha\rangle=
\sum_{\delta\in S}\left(||\delta \alpha||^2+||\delta^*\alpha||^2\right)
\]
so all the components $\delta$ and $\delta^*$, for all $\delta\in S$, vanish on $\alpha$.
This proves the remaining identities.
\end{proof}

We will consider the spaces of \textit{$\delb$-$\mu$-harmonic forms} given by the intersections
\[\Hh_\delb^{*,*}\cap \Hh_\mu^{*,*}.\]
By the previous lemma, these are identified with the kernel of the self-adjoint elliptic operator 
given by
$\Delta_\delb+\Delta_\mu$.
We denote by
\[\ell^{p,q}:=\dim \left(\Hh_\delb^{p,q}\cap \Hh_\mu^{p,q}\right)\]
the dimensions of these spaces.
Note that in the integrable case, these are just the Hodge numbers of the manifold.

\section{Almost K\"ahler identities}\label{SecAKIds}

An \textit{almost K\"ahler manifold} is by definition an almost Hermitian manifold such that the associated $(1,1)$-form is closed. Equivalently, an almost K\"ahler manifold is a symplectic manifold with a compatible metric, i.e. there is an induced almost complex structure which is compatible with both the symplectic form and the metric. 
 
On an almost K\"ahler manifold, the so-called \textit{K\"{a}hler identities}, involving the differential
operator $\delb$, the Lefschetz operator $L$, and their complex conjugates and adjoints, 
hold just as in the case of K\"{a}hler manifolds.
These are proven in Weil's book \cite{Weil}, where indeed, the integrability condition is not used.
More recent references often prove the K\"{a}hler identities by
reducing the proof to a computation in $\CC^m$, thus restricting to the integrable setting 
(see for instance \cite{GrHa}, \cite{Voisin}, see also Remark 3.1.14 of \cite{Huy}).

In this section, we retake Weil's approach and give analogous identities involving the operators 
$\mu$ and $\mub$. From these, we obtain several commutation relations
involving the four components $\mub$, $\delb$, $\mu$ and $\del$ of the differential,
as well as various relations between Laplacians.

In what follows we define the graded commutator of operators $A$ and $B$ by
\[
[A,B] = AB-(-1)^{\textrm{deg(A)deg(B)}}BA
\]
 where $deg(A)$ denotes the total degree of $A$.

\begin{prop} \label{AKidentities}
 For any almost K\"ahler manifold the following identities hold:
\begin{enumerate}[itemsep=2mm]
\item $[L,\mub] = [L,\mu] = 0$ and $[\Lambda,\mub^*] = [\Lambda, \mu^*] = 0$.
\item $[L,\delb] = [L,\del] = 0$ and $[\Lambda,\delb^*] = [\Lambda, \del^*] = 0$.
\item $[L, \mub^*] = i \mu $, $[L, \mu^* ] = - i \mub$ and $[\Lambda , \mub] = i \mu^*$, $[\Lambda , \mu] = - i \mub^*$.
\item $[L, \delb^*] = -i \del $, $[L, \del^*] =  i \delb$ and $[\Lambda , \delb] = -i \del^*$, $[\Lambda , \del] = i \delb^*$.
\end{enumerate}
\end{prop}
 
\begin{proof}
Since $\omega \in \cA^{1,1}$ is $d$-closed we have $\mub \, \omega = 0$, and since $\mub$ is a derivation, $[\mub,L] =0$. The remaining cases in the first statement follow by taking complex conjugates and adjoints, since $\omega$ is real.
The proof for the second statements is identical and well known.

For the third and fourth statements, using the primitive decomposition of the exterior algebra of the manifold, and the fact that $d\omega=0$, it is well known that 
\[
[\Lambda , d ] = \star \, \mathbb{I}^{-1} \, d \, \mathbb{I} \, \star ,
\]
where $\mathbb{I}$ is the operator that acts on $(p,q)$-forms by multiplication by $i^{p-q}$ (c.f. \cite{Huy} Proposition 3.1.12, p.121-122).  In bidegree $(p,q)$ we have $\mathbb{I}_{p,q}^{-1} = (-1)^{p-q} \mathbb{I}_{p,q}$, so conjugating an operator of bidegree $(r,s)$ by $\mathbb{I}$ acts by multiplication by $(-i)^{r-s}$. Then using $\bar \delta^* = - \star \delta \star$ for $\delta= \mub, \delb, \del,$ and $\mu$, it follows that
\begin{align*}
\star \, \mathbb{I}^{-1} \, \mub \, \mathbb{I} \, \star &= i \mu^*, \\
\star \, \mathbb{I}^{-1} \, \delb \, \mathbb{I} \, \star &= -i\del^*, \\
\star \, \mathbb{I}^{-1} \, \del \, \mathbb{I} \, \star &=  i\delb^*, \\
\star \, \mathbb{I}^{-1} \, \mu \, \mathbb{I} \, \star &=-i\mub^*.
\end{align*}
Then $d = \mub + \delb + \del + \mu$ implies the third and fourth statements involving $\Lambda$. The statements involving $L$ follow by taking adjoints.
\end{proof}

\begin{rmk} 

The above almost K\"{a}hler identities were also previously established by de Bartolomeis and Tomassini (c.f. \cite{BaTo}, Lemma $3.4$) by first
proving $[d, \Lambda ]= d^s$, where  $d^s$ is the symplectic adjoint of $d$. 
In this way, they also obtain the first two identities in Proposition \ref{AKidentities2} below, (c.f. \cite{BaTo}, Lemma $3.7$).
\end{rmk}

We next deduce the following relations concerning the various components of $d$ and their adjoints.
It is helpful to use the graded Jacobi identity:
\[
[A,[B,C]] = [[A,B],C] + (-1)^{\textrm{deg(A)deg(B)}} [B,[A,C]].
\]

\begin{prop} \label{AKidentities2}
For any almost K\"ahler manifold the following identities hold:
\begin{enumerate}[itemsep=2mm]
\item $[\mub, \mu^*] = [\mu, \mub^*] = 0 $.
\item $[\mub, \del^*] = [\delb,\mu^*]$ and $[\mu, \delb^*] = [\del,\mub^*]$.
\item $ [\del, \delb^*] = [ \mub^*,\delb] +  [\mu,\del^*] $ and $ [\delb, \del^*] = [ \mu^*,\del] + [\mub,\delb^*] $.
\end{enumerate}
\end{prop}

\begin{proof}
For the first statement
\[
[\mub,\mu^*] = i [\mub,[\mub, \Lambda]] = 0,
\]
and the second follows by conjugation or adjoint.

Next we have
\[
[\mub, \del^*] = i [\mub, [\Lambda,\delb]] = i  [[\mub,\Lambda],\delb] + i [\Lambda, [\mub, \delb]]
\]
by the graded Jacobi identity. Since $[\mub, \delb]=0$, and $[\mub,\Lambda] = -i \mu^*$, this becomes
\[
[\mub, \del^*] = [\mu^*, \delb] = [\delb, \mu^*].
\]
Then $[\mu, \delb^*] = [\del,\mub^*]$ by conjugation or adjoint. 

The next two claims are also equivalent by conjugation and adjoint. We'll prove the first one. First, using $[\Lambda, \del] = i \delb^*$, we compute
\[
[\del, \delb^*] = -i [ \del , [\Lambda, \del]] = i [ \del^2, \Lambda] = i [\Lambda, [\mu ,\delb]]
\]
where in the last step we used $[\mu ,\delb]+ \del^2 = 0$. By the graded Jacobi identity,
\[
[\del, \delb^*] = i [[\Lambda, \mu] ,\delb] + i [\mu, [\Lambda ,\delb]].
\]
Now using $[\Lambda, \mu]= -i \mub^*$ and $ [\Lambda ,\delb] = -i \del^*$ the result follows.
\end{proof}

Note that in the first statement $[\mub, \mu^*] = 0$ is a zeroth-order (metric-dependent) condition which obstructs an almost Hermitian manifold from being symplectic, whereas $d \omega =0$ is a first order (metric-independent) condition. We next deduce several relations concerning various Laplacians. 

\begin{prop} \label{AKDeltas}
For any almost K\"ahler manifold the following identities hold:
\begin{enumerate}[itemsep=2mm]
\item $\Delta_{\mub +\mu} = \Delta_\mub + \Delta_\mu$.
\item $ \Delta_\delb + \Delta_\mu = \Delta_\del + \Delta_\mub$.
\item $\Delta_d = 2 \left( \Delta_\delb + \Delta_\mu+  [\mub, \del^*]  + [\mu, \delb^*] +
[\del, \delb^*] + [\delb, \del^*]  \right) $.
\end{enumerate}
\end{prop}

\begin{proof}
The first claim follows by direct calculation using $[\mub, \mu^*] = [\mu, \mub^*] = 0 $.
For the second claim, by Proposition \ref{AKidentities},
\begin{align*}
\Delta_\mub=\mub\mub^*+\mub^*\mub &= i \left(\mub[\Lambda,\mu]+[\Lambda,\mu]\mub\right)\\
&=i \left(\mub \Lambda\mu -\mub\mu \Lambda+\Lambda\mu \mub -\mu \Lambda \mub\right)
\end{align*}
and similarly
\begin{align*}
\Delta_\mu=\mu\mu^*+\mu^*\mu&=-i \left(\mu[\Lambda,\mub]+[\Lambda,\mub]\mu\right)\\
&=-i \left(\mu \Lambda\mub -\mu\mub \Lambda+\Lambda\mub \mu -\mub \Lambda \mu\right)
\end{align*}
so that
\begin{align*}
 \Delta_\mub-\Delta_\mu &= i \left(\Lambda(\mu\mub+\mub\mu) - (\mu\mub+\mub\mu)\Lambda\right)\\
 &=-i \left(\Lambda(\del\delb+\delb\del) - (\del\delb+\delb\del)\Lambda\right) \\
& = \Delta_\delb-\Delta_\del.
\end{align*}
The last equality follows from a similar calculation as is done above for $\Delta_\mub-\Delta_\mu$.

Finally, expanding $\Delta_d = [d,d^*]$ and using $d = \mub + \delb + \del + \mu$, we have
\begin{align*}
\Delta_d & = \Delta_\mub + \Delta_\delb + \Delta_\del + \Delta_\mu \\
& + [\mub, \delb^*] + [\mub, \del^* ]  + [\mub,\mu^* ] \\
& + [\delb, \mub^*] + [\delb, \del^* ]  + [\delb,\mu^* ] \\
& + [\del, \mub^*] + [\del, \delb^* ]  + [\del,\mu^* ] \\
& + [\mu, \mub^*] + [\mu, \delb^* ]  + [\mu,\del^* ],
\end{align*}
so the final statement follows using Proposition \ref{AKidentities2} and the previous part.
\end{proof} 

We have one more set of useful relations, which are related to hard Lefschetz duality.

\begin{cor} \label{LDeltaetc}
For any almost K\"ahler manifold the following identities hold:
\begin{enumerate}[itemsep=2mm]
\item  $[L, \Delta_\delb] = [L, \Delta_\mub]  = - [L, \Delta_\del] = - [L, \Delta_\mu]=  -i[\delb,\del]= i [\mub,\mu] $.
\item $[\Lambda, \Delta_\delb] = [\Lambda, \Delta_\mub]  =  -[\Lambda, \Delta_\del] = - [\Lambda, \Delta_\mu]=  -i[\delb^*,\del^*]= i [\mub^*,\mu^*] $.
\end{enumerate}
\end{cor}

\begin{proof}
Using Propostion \ref{AKidentities}, we calculate
\[
  [L, \Delta_\delb]  = [\delb, [L,\delb^*]]= -i[\delb,\del]= i [\mub,\mu] =  [\mub, [L,\mub^*]]=[L, \Delta_\mub]  ,
\]
and all remaining relations follow from taking conjugates or adjoints.
\end{proof}

In conclusion, for any almost K\"ahler manifold, there is a $\Z_2$-graded Lie algebra of operators acting on the $(p,q)$-forms, generated by eight odd operators 
\[
\delb, \del,\mub,  \mu, \delb^*, \del^*, \mub^*,  \mu^*
\]
and even degree operators $L, \Lambda, H$, from which all relations can be deduced from those given above. In the integrable case, this reduces to the so-called $N=2$ supersymmetry algebra of a K\"ahler manifold (see for instance
\cite{Zumino}, \cite{AF}, \cite{HKLR}),
also referred to as the $N=(2,2)$ supersymmetry algebra in \cite{FGR}.

\section{Topological and geometric consequences in the compact case}\label{Section_consequences}

From the almost K\"ahler identities and the symmetries of the Laplacians, we deduce in this section several 
results giving combined geometric/topological restrictions in the compact case. 
Recall that $\ell^{p,q}=\dim \Hh_\delb^{p,q}\cap \Hh_\mu^{p,q}$ denotes the dimension of the space of $\delb$-$\mu$-harmonic forms of type $(p,q)$. First, we have the following identities.

\begin{thm}\label{AKharmonicequal}
 For any compact almost K\"ahler manifold of dimension $2m$ 
 and for all $(p,q)$, we have identities
 \[\Hh_d^{p,q}=\Hh_\delb^{p,q}\cap \Hh_\mu^{p,q}=\Hh_\del^{p,q}\cap\Hh_\mub^{p,q}.\]
 Also, the following dualities hold:
 \begin{enumerate}
  \item \emph{(Complex conjugation)}. We have equalities
  \[\Hh_\delb^{p,q}  \cap \Hh_\mu^{p,q}  =  \Hh_\delb^{q,p}  \cap \Hh_\mu^{q,p}.\]
\item \emph{(Hodge duality)}. The Hodge $\star$-operator induces isomorphisms
\[
\star : \Hh_\delb^{p,q}  \cap \Hh_\mu^{p,q} \to \Hh_\delb^{m-q,m-p}  \cap \Hh_\mu^{m-q,m-p}.
\]
\item \emph{(Serre duality)}. There are isomorphisms
\[\Hh_\delb^{p,q}  \cap \Hh_\mu^{p,q} \cong  \Hh_\delb^{m-p,m-q}  \cap \Hh_\mu^{m-p,m-q}.\]
 \end{enumerate}
 \end{thm}
\begin{proof}
The identity
\[\Hh_d^{p,q}=\Hh_\delb^{p,q}\cap \Hh_\mu^{p,q}=\Hh_\del^{p,q}\cap\Hh_\mub^{p,q}.\]
follows from 
Lemma \ref{intersections} together with the almost K\"{a}hler identity 
\[\Delta_\delb+\Delta_\mu=\Delta_\del+\Delta_\mub\]
of Proposition \ref{AKDeltas}.
This identity implies that the various dualities of Lemma \ref{dualities}
are also satisfied for the spaces $\Hh_\delb^{*,*}\cap\Hh_\mu^{*,*}$ of $\delb$-$\mu$-harmonic forms.
\end{proof}

Let $b^k:=\dim H^k(M,\CC)$ denote the Betti numbers of a manifold $M$. 
Theorem \ref{AKharmonicequal} implies that,
for a compact almost K\"{a}hler manifold, the 
numbers $\ell^{*,*}$ are bounded above by the Betti numbers and
satisfy the Hodge diamond-type symmetries,
\[\ell^{p,q}=\ell^{q,p}=\ell^{m-p,m-q}=\ell^{m-q,m-p}.\]

Next, note that all powers of an almost K\"{a}hler form are pure $d$-harmonic.

\begin{lem}\label{wharmonic}
If $\omega$ is an almost K\"{a}hler form on a compact manifold, then $\omega^k \in \Hh_d^{k,k}$ for all $k \geq 0$.
\end{lem}

\begin{proof}
Since $d\omega=0$ we also have $d(\omega^k)=0$ for all $k\geq 0$. Using the fact that $\star w^k=w^{n-k}$, we have 
$d^*(w^k)=0$ for all $k\geq 0$.
\end{proof}

With this, we obtain the following inequalities:

\begin{thm} \label{AKbounds}
For any compact almost K\"ahler manifold of dimension $2m$, the following is satisfied:
\begin{enumerate}
\item In odd degrees, we have 
\[\sum_{p+q=2k+1}\ell^{p,q}=2\sum_{0\leq p\leq k}\ell^{p,2k+1-p}\leq b^{2k+1}.\]
\item In even degrees, we have \[\sum_{p+q=2k}\ell^{p,q}=2\sum_{0\leq p<k}\ell^{p,2k-p}+\ell^{k,k}\leq b^{2k},\]
with $\ell^{k,k}\geq 1$ for all $k\leq m$.
\end{enumerate}
\end{thm}
\begin{proof}
The inclusion 
$\Hh_d^{p,q}\subseteq \Hh_d^{p+q}$
together with the equalitities $\Hh_d^{p,q}=\Hh_\delb^{p,q}\cap \Hh_\mu^{p,q}$ of Theorem \ref{AKharmonicequal}
give inequalities
\[\sum_{p+q=k} \ell^{p,q}\leq b^{k}\text{ for all }k\geq 0.\]
The inequalities now follow from the dualities of Theorem \ref{AKharmonicequal} together with Lemma \ref{wharmonic}, which ensures that $\ell^{k,k}\geq 1$.
\end{proof}

We deduce various combined geometric/topological obstructions to the almost K\"ahler condition. 
An immediate corollary is:
\begin{cor}\label{cor1}
For any compact almost K\"ahler manifold, if $\ell^{p,q}  \neq 0$ for some $p \neq q$, then 
$b^{p+q} \geq  2$ if $p+q>0$ is even, and $b^{p+q}\geq 3$ if $p+q>0$ is odd.
\end{cor}

\begin{ex}
For any almost K\"{a}hler structure on a $4$-dimensional manifold with the same cohomology as $\CC\PP^2$, the numbers $\ell^{*,*}$ are given by 
\[
\xymatrix@C=.1pc@R=.1pc{
 &&1\\
 &0&&0\\
 0&&1&&0\\
 &0&&0\\
 &&1
}
\] 
Likewise, for any almost K\"{a}hler structure on a $4$-dimensional manifold with the same cohomology as $S^2\times S^2$,  the numbers $\ell^{*,*}$ are given by 
\[
\xymatrix@C=.1pc@R=.1pc{
 &&1\\
 &0&&0\\
 0&&k&&0\\
 &0&&0\\
 &&1
}
\] 
where $k\in\{1,2\}$.
Using intersection pairings we see using Theorem \ref{4mnfldIndex} below that in fact $k=1$.
 \end{ex}

By restricting to $1$-forms, we obtain the following metric-independent statements.
Denote by 
\[\Omega^p_\delb:=\Ke(\delb)\cap \cA^{p,0}\]
the space of \textit{holomorphic $p$-forms}. As was previously known for
compact K\"{a}hler manifolds, the space of holomorphic $1$-forms coincides with the space of $d$-harmonic forms of pure type $(1,0)$. Indeed, by bidegree reasons together with Theorem \ref{AKharmonicequal}, we have 
\[\Omega^1_\delb=\Hh_\delb^{1,0}=\Hh_\delb^{1,0}\cap\Hh_\mu^{1,0}=\Hh_d^{1,0}.\]
This immediately gives:

 \begin{cor}\label{cor2}
A necessary condition for a compact almost complex manifold to admit a compatible symplectic form is that
\[2\dim\Omega^{1}_\delb\leq b^{1}.\]
\end{cor}

In particular, compact simply connected almost K\"{a}hler manifolds
(and actually any compact almost K\"{a}hler manifold with $b^1\leq 1$) have no holomorphic $1$-forms.

We next extend Chen's results \cite{Chen}, \cite{Chen2} on fundamental groups of 
K\"{a}hler manifolds, 
to the almost K\"{a}hler case.

\begin{lem}\label{Chensolv}
If $M$ is a compact connected almost K\"{a}hler manifold and if
\[\dim \Omega_\delb^1>\dim \Omega_\delb^2+1,\]
then $\pi_1(M)$ contains a free subgroup of rank $\geq 2$. In particular,
$\pi_1(M)$
is not solvable.
\end{lem}

\begin{proof}
Theorem 3.2 of \cite{Chen} states that if
there exist closed $1$-forms $\alpha$ and $\beta$ such that $\alpha \wedge \beta=0$ and $[\alpha]$ and $[\beta]$ are linearly independent classes in $H^1(M,\CC)$, then $\pi_1(M)$ is not solvable. In Theorem 4.1 of \cite{Chen2}, it is shown that the same conditions imply that $\pi_1(M)$ contains a free subgroup of rank $\geq 2$.
We adapt Chen's Corollary on K\"{a}hler manifolds. Let $\alpha_1,\cdots,\alpha_r$ generate $\Omega^1_\delb$. Since by assumption we have $\dim \Omega_\delb^2<r-1$,
the forms $\alpha_1 \wedge \alpha_i$, for $i>1$, must be linearly dependent. Hence there is a holomorphic $1$-form $\alpha$ such that $\alpha_1 \wedge\alpha=0$, and $\alpha_1$ and $\alpha$ are linearly independent in $\Omega^1_\delb$.
Since $\Omega^1_\delb= \Hh_d^{1,0}$ and the map \[\Hh_d^{1,0}\to H^{1}_{\dR}(M,\CC)\] is injective, the classes $[\alpha_1]$ and $[\alpha]$
are linearly independent.
\end{proof}

\begin{ex}
Take for instance a solvable finitely presented group $G$. Then $G$ can be realized as the fundamental group of a symplectic $4$-manifold $M$ by a result of Gompf \cite{Gompfreal}. The above result tells us that for any compatible almost complex structure, the inequality 
$\dim \Omega_\delb^1 \leq \dim \Omega_\delb^2+1$ is satisfied.
\end{ex}

\section{Lefschetz duality, decomposition, and Hodge index theorem}\label{SecLef}

We next prove a generalized Lefschetz duality and Lefschetz decomposition for compact almost K\"{a}hler manifolds.
We also give a generalized Hodge index theorem in the four dimensional case.

\begin{thm}
[Generalized Hard Lefschetz Duality] \label{HLef}
\label{Lefduality}
For any compact almost K\"ahler manifold of dimension $2m$, the operators $\{L,\Lambda, H = [L,\Lambda] \}$ define a 
finite dimensional representation of $\mathfrak{sl}(2,\CC)$ on  
\[
\bigoplus_{p,q \geq 0} \Hh_d^{p,q} = \bigoplus_{p,q \geq 0} \Hh_\delb^{p,q} \cap \Hh_\mu^{p,q} = 
\Ke(\Delta_\delb+\Delta_\mu).
\]
Moreover, for all $0 \leq p \leq k\leq m$,
\[
L^{m-k} :  
\Hh_d^{p,k-p} 
\stackrel{\cong}{\longrightarrow} 
\Hh_d^{p+m-k,m-p}
\]
are isomorphisms.
\end{thm}

\begin{proof} 
For any almost Hermitian manifold of dimension $2m$ there are isomorphisms
\[
L^{m-k} : \cA^{p,k-p} \stackrel{\cong}{\longrightarrow} \cA^{p+m-k,m-p}
\]
for every $0 \leq p \leq m$ and all $p \leq k \leq m$. 
 By Corollary \ref{LDeltaetc}, $[L, \Delta_\delb + \Delta_\mu]  = 0$ and $[\Lambda, \Delta_\delb + \Delta_\mu]  = 0$,
 so $L$ and $\Lambda$ preserve $\Hh_\delb \cap \Hh_\mu$.
It follows the maps 
\[
L^{m-k} :  
\Hh_d^{p,k-p} 
\stackrel{\cong}{\longrightarrow} 
\Hh_d^{p+m-k,m-p}
\]
are well defined, and are injective, since they are isomorphisms before restricting the domain. By Hodge duality of Lemma \ref{dualities}, the domain and codomain have the same dimension, so the map is an isomorphism. 
\end{proof}

\begin{rmk}
Benson and Gordon showed that if a symplectic nilmanifold $M$ satisfies that $L: H^1 \to H^{2n-1}$ is an isomorphism, then $M$ is a torus \cite{BG}. On the other hand, there are many non-toral symplectic nilmanifolds, so the above generalized Lefschetz duality has a large family of non-trivial examples which are computable.
\end{rmk}

\begin{rmk}[Comparison with symplectic Hodge theory]
In \cite{Brylinski}, Brylinski proposed a Hodge theory for compact symplectic manifolds, by introducing a 
symplectic Hodge star operator $\star_s$, defined using the symplectic form.  The space of \textit{symplectically-harmonic $k$-forms} is
\[\Hh_{sym}^k:=\cA^k\cap \Ke(d)\cap \Ke(d_s)\]
where $d^s = \star_s \circ d \circ \star_s$.
Brylinski showed that in an almost K\"{a}hler manifold, a form of pure type $(p,q)$ is in $\Hh_{sym}^{p+q}$ if and only if it is in $\Hh_d^{p+q}$. This follows from Brylinski's formula for almost K\"ahler manifolds:
if $\alpha\in\cA^{p,q}$ then 
$\star_s(\alpha)=i^{p-q}\star(\alpha).$
In particular, by  (1) of Theorem  \ref{AKharmonicequal}, when restricted to forms of pure type, all notions of harmonics agree:
\[\Hh_\delb^{p,q}\cap \Hh_{\mu}^{p,q}=\Hh_d^{p,q}=\Hh_{sym}^{p+q} \cap \cA^{p,q}.\]
This gives an inclusion 
\[\bigoplus_{p+q=k}\Hh_d^{p,q}\hookrightarrow \Hh_{sym}^{p+q}\]
which is strict in general. Indeed, Yan \cite{Yan} showed that for $k=0,1,2$, every cohomology class has a symplectically-harmonic representative. This is not true for $\delb$-$\mu$-harmonic forms (see the example in \ref{KT} below).
\end{rmk}

As a consequence of Theorem \ref{Lefduality}, we have the following generalized Lefschetz decomposition, 
which sheds some light on the relative scarcity of examples of compact symplectic $2m$-manifolds whose Betti numbers are not monotone, i.e. $b_j < b_{j-2} $ for some $j \leq  m$.

\begin{cor} \label{PrimDecomp}
For any compact almost K\"ahler manifold of dimension $2m$, and any $p,q$
we have an orthogonal direct sum decomposition
\[
 \Hh_d^{p,q}  = \bigoplus_{j \geq 0} L^j \left(\Hh_d^{p-j,q-j}\right)_{\textrm{prim}}
\]
where 
\[
\left(\Hh_d^{r,s}   \right)_{\textrm{prim}} :=  \Hh_d^{r,s}    \cap \Ke \Lambda.
\]
\end{cor}

This may be deduced directly, as a finite dimensional representation of  $\mathfrak{sl}(2,\CC)$. It follows that one can also express the dimensions of the primitive spaces in terms of successive differences of the numbers $\ell^{p,q}$. 

We have inequalities
\[
\ell^{p,q} \leq \ell^{p+1,q+1} \leq \dots \leq \ell^{p+j,q+ j}
\]
for all $0 \leq p, q \leq m$ and $p + q + 2j \leq m$. 
In particular, note that if $b^k = 0$ for some $k \leq m$, then $ \ell^{p,q}=0$ for all $p+q \leq k$ with $p+q \equiv k \mod 2$.

We give some applications below.

\begin{rmk}
The Hodge-Riemann pairing yields the analogous bilinear Hodge-Riemann relations. 
 Namely,  for any compact almost K\"ahler manifold  of dimension $2m$, if 
$\alpha \in \left(\Hh_d^{p,q}\right)_{\textrm{prim}}$ is a non-zero form,  then
\[
i^{p-q}  Q(\alpha, \bar \alpha) >0,
\]
where the Hodge-Riemann pairing $Q$ is given by 
\[
Q(\alpha ,\beta)  := (-1)^{\frac{(p+q)(p+q-1)}{2})} \int_M \alpha \wedge \beta \wedge \omega^{m-p-q}.
\]
This follows from Corollary \ref{PrimDecomp}, since any primitive $(p,q)$-form $\alpha$ on any almost Hermitian manifold, 
\[
 i^{p-q}  (-1)^{\frac{(p+q)(p+q-1)}{2})} \alpha \wedge \bar \alpha \wedge \omega^{m-p-q} 
\]
is a positive multiple of the volume form $\frac{\omega^m}{m!}$ at any point where $\alpha$ is non-vanishing (c.f. \cite{Huy}, p.39 Corollary 1.2.36). 
\end{rmk}

For the remainder of this section, we restrict to dimension $4$, where there is a version of the Hodge Index Theorem for almost K\"{a}hler $4$-manifolds. First, for any almost Hermitian $4$-manifold, we have the following result:

\begin{lem}\label{h20}
For any almost Hermitian $4$-manifold with a non-integrable almost complex structure, $\Hh_d^{2,0}=0$.  In particular, for any compact non-integrable almost K\"ahler $4$-manifold we have $\ell^{2,0}=\ell^{0,2}=0$.
\end{lem}
\begin{proof}
Since $\Hh_d^{p,q}$ injects into the cohomology groups $H_J^{p,q}$, triviality of $\Hh_d^{2,0}$ follows directly from Lemma 2.12 of \cite{DrLiZh}, where it is shown that $H^{2,0}_J=0$. We include an argument for completeness.
If the almost complex structure is not integrable, there is an open set $U$ on which $\mu: T^{0,1} M \to \left( T^{1,0} M \wedge T^{1,0} M \right)$ is pointwise nonzero, and therefore pointwise surjective for all points in $U$, since $M$ is a four manifold.
Therefore, 
\[
\mu^* : \wedge^2 T^{1,0}M  \to T^{0,1} M
\]
is pointwise injective on $U$. 
If $\eta \in \Hh_d^{2,0}$, then $\mu^* \eta= 0$, which implies $\eta$ is zero on $U$, hence $\eta$ is zero everywhere \cite{Bar}.
The identities $\ell^{2,0}=\ell^{0,2}=0$ follow from Theorem \ref{AKharmonicequal}.
 \end{proof}

 \begin{thm}
 [Generalized Hodge Index Theorem, c.f. \cite{HZ}] \label{4mnfldIndex}
 For any compact $4$-dimensional almost K\"ahler manifold $M$, we have
\[
\ell^{1,1}=  b_2^- + 1 \text{ and } b_2^+\geq 1,\]
 where the intersection pairing on $H^2(M;\CC)$ has index $(b_2^+,b_2^-)$. 
 \end{thm}

A proof of the equality $\ell^{1,1}=  b_2^- + 1$ first appeared in the preprint \cite{HZ}, following the inequality $\ell^{1,1} \leq b_2^- + 1$ first established in a preprint of the present paper.  The proof of equality in \cite{HZ}  follows arguments previously developed in \cite{DrLiZh}, \cite{LZ}, and \cite{FTAC}, and strongly depends on Theorem \ref{AKharmonicequal}. We include a proof here for completeness.

\begin{proof} 
 Recall that on a 4-manifold, the Hodge star operator induces a decomposition  on the space of $2$-forms into self-dual $\cA^+$ and anti-self-dual $\cA^-$ spaces (as eigenspaces $\pm1$). Since the Laplacian commutes with the Hodge star operator, the $d$-harmonic $2$-forms decompose as
\[\Hh^2_d=\Hh_d^+\oplus \Hh_d^-,\text{ where }\star \Hh^{\pm}_d=\pm\Hh^{\pm}_d.\]
The intersection pairing is positive definite on $\Hh_d^+$, and negative definite on $\Hh_d^-$. 
Since $\omega \in \Hh_d^+$, it follows that $b_2^+ \geq 1$.
We now show that
\[\Hh_d^{1,1}= \Hh_d^- + \CC[\omega]\]
by checking both inclusions.
First, by Lemma \ref{wharmonic} we have  $\omega \in \Hh_d^{1,1}$.
Furthermore, any form in $\Hh_d^-$ must be of pure type $(1,1)$ since a local computation shows there is an orthogonal sum
$\cA^{1,1} = \cA^- \oplus \CC[\omega]$. 

To prove the converse containment, we write any $d$-harmonic $(1,1)$-form uniquely as 
$\eta=\gamma+f\omega$, where $f$ is a complex valued function and $\gamma$ is anti-self dual.
Since $\star$ commutes with $\Delta_d$, it follows that both $\gamma$ and $f\omega$ are $d$-harmonic.
Since $L$ is injective in degree one and $L(df)=df\wedge \omega=0$, it follows that
$f$ is constant. This proves the remaining inclusion $\Hh_d^{1,1}\subseteq \Hh_d^- + \CC[\omega]$.
\end{proof}

\begin{rmk}
In the case that the containment 
$\bigoplus_{p+q=2k} \Hh_d^{p,q} \subseteq \Hh_d^{2k}$ is an equality, such as for K\"ahler manifolds,
one also has the second identity $2\ell^{2,0}+1=b_2^+$.
\end{rmk}

\begin{cor}\label{metricindep}
For any compact almost K\"ahler $4$-manifold
all of the numbers $\ell^{p,q}$ are metric-independent among all almost 
K\"ahler metrics that are compatible with the given almost complex structure.
\end{cor}
\begin{proof}
We may as well assume the almost complex structure is not integrable. 
By Lemma \ref{h20} we have $\ell^{2,0}=\ell^{2,0}=0$.
By Theorem \ref{4mnfldIndex}, $\ell^{1,1}$ is a topological invariant.
Lastly, note that
\[\Hh_d^{1,0}=\Hh_\delb^{1,0}\cap \Hh_\mu^{1,0}= \Ke(\delb)\cap \cA^{1,0}\]
which is metric independent. The dualities
$\ell^{0,1}=\ell^{1,0}=\ell^{2,1}=\ell^{1,2}$
give metric-independence for all remaining bidegrees.
\end{proof}

\begin{rmk}
Arguing as in the corollary above we obtain 
metric-independence for all numbers $h^{p,q}_\delb:=\dim \Hh_\delb^{p,q}$ except for $h^{0,1}_\delb=h^{2,1}_\delb$.
Indeed, for bidegree reasons we have $h^{1,1}_\delb=\ell^{1,1}$ and 
$\Hh_\delb^{p,0}=\Ke(\delb)\cap \cA^{p,0}$ which, together with the Hodge dualities $h_\delb^{p,0}=h_\delb^{2-p,2}$
give metric-independence for $h^{1,0}_\delb=h^{1,2}_\delb$ and $h_\delb^{2,0}=h^{0,2}_\delb$.
In \cite{HZ}, Holt-Zhang construct an example
of an almost complex structure whose $h_\delb^{0,1}$ varies with
different choices of Hermitian metrics. However, it remains an open question to know whether $h_\delb^{0,1}$ is metric-independent among almost 
K\"ahler metrics only.
\end{rmk}

\section{Applications}\label{SecNil}

In general it is very difficult to compute the dimension of the kernel of a self adjoint elliptic operator. In this section we show the above theory can be used to compute harmonic spaces,
and give other applications involving the above obstructions to almost K\"ahler structures. For some examples we use nilmanifolds, which are a quotient of a nilpotent real Lie group $G$ by an integral subgroup $\Gamma$. For this purpose, we first briefly describe how the above theory applies in that setting.

An almost complex structure on the Lie algebra $\mathfrak{g}$ of $G$ defines a left-invariant almost complex structure on the associated nilmanifold.
Also, it defines a bigrading
on the Chevalley-Eilenberg dg-algebra $\cA^*_{\mathfrak{g}_\CC}$ associated to the complexification $\mathfrak{g}_\CC$
of $\mathfrak{g}$ and $M$ inherits an almost complex structure. 
The algebra $\cA^*_{\mathfrak{g}_\CC}$ may be regarded as the complex algebra of $\Gamma$-invariant forms on $G$ and it includes into the complex algebra of forms of $M$ via a quasi-isomorphism compatible with bigradings.
The theory of harmonic forms for almost complex nilmanifolds is developed in \cite{CWDol}.
For our purposes here, it suffices to note that, given a left-invariant Hermitian metric, we have an inclusion
\[{}^L\Hh_\delta^{p,q}:=\Ke(\Delta_\delta)\cap\cA^*_{\mathfrak{g}_\CC}\subseteq \Hh_\delta^{p,q}\]
of left-invariant $\delta$-harmonic forms into all $\delta$-harmonic forms, and that
the almost K\"{a}hler package is equally valid on these subspaces.
We see below that, in some situations, the dimensions of the left-invariant harmonics, and the bounds offered by 
the almost K\"{a}hler package and topology, provide sufficient information to compute the dimensions $\ell^{p,q}$ of the true harmonic spaces.

\subsection{Kodaira-Thurston manifold} \label{KT}
The Kodaira-Thurston manifold was originally studied by Kodaira as a complex manifold  \cite{Kod}, and by Thurston as the first example of a symplectic manifold which is non-K\"ahler, \cite{Thur}. We follow the presentation in \cite{BM} as a nilmanifold, giving here an almost K\"ahler structure.

The Kodaira-Thurston manifold is the $4$-dimensional nilmanifold defined as the quotient 
\[
\mathrm{KT}=H_\Z \times \Z \backslash H \times \R
\]
 where $H$ is the $3$-dimensional Heisenberg Lie group, and $H_\Z$ is the integral subgoup. The Lie algebra is spanned by $X,Y,Z,W$ where the only non-zero bracket is $[X,Y] = -Z$. It has Betti numbers $b^1=3$ and $b^2=4$ and topological index $(2,2)$.
 We remark that $L: H^1 \to H^3$ is not an isomorphism, i.e. the hard Lefschetz duality does not hold on cohomology.

\begin{prop}
For any left-invariant almost K\"{a}hler structure on the Kodaira-Thurston manifold, the numbers $\ell^{p,q}$ are given by 
\[
\xymatrix@C=.1pc@R=.1pc{
 &&1\\
 &1&&1\\
 0&&3&&0\\
 &1&&1\\
 &&1
}
\]
\end{prop}

\begin{rmk}
To contrast, for any left-invariant integrable  structure on the Kodaira-Thurston manifold,
the numbers $\ell^{p,q}=h_\delb^{p,q}:=\dim \Hh_\delb^{p,q}$ are given by
\[
\xymatrix@C=.1pc@R=.1pc{
 &&1\\
 &2&&1\\
 2&&4&&2\\
 &1&&2\\
 &&1
}
\]
\end{rmk}

\begin{proof}
By Lemma \ref{h20} we have $\ell^{2,0}=\ell^{0,2}=0$, and by Theorem \ref{4mnfldIndex} we have $\ell^{1,1}=3$.
By Theorem \ref{AKbounds} we know that $2\ell^{1,0}\leq b^1=3$ and so $\ell^{1,0}\in \{0,1\}$.
To prove that $\ell^{1,0}\neq 0$, note that for any almost K\"{a}hler manifold we may identify the space 
$\Hh_d^{1,0}=\Ke(\delb)\cap \cA^{1,0}$ with the Dolbeault cohomology 
\[H^{1,0}_{\Dol}:=\Ke(\mub)\cap\Ke(\delb)\cap \cA^{1,0}\]
introduced in \cite{CWDol}. Note that this identification is special to the almost K\"{a}hler case since, for a general almost complex structure,
\[\Ke(\delb)\cap\cA^{1,0}\neq \Ke(\delb)\cap \Ke(\mub)\cap\cA^{1,0}.\]
The Dolbeault cohomology has the advantage that there is a Fr\"{o}licher-type spectral sequence, 
with $E_1=H_{\Dol}$, and which converges to complex de Rham cohomology. 

We will not need the full details and machinery, but rather compute just one page and use the convergence to complex de Rham cohomology to conclude that there is a non-zero left invariant $(1,0)$-form which is in $H^{1,0}_{\Dol}$, and therefore $\ell^{1,0} = 1$. First, we can compute the groups 
\[
H_\mub^{p,q}:= {\Ke  \left( \mub: \cA^{p,q}_L \to \cA^{p-1,q+2}_L \right) \over \Img  \left(\mub: \cA^{p+1,q-2}_L \to \cA^{p,q}_L \right)}
\] 
for left invariant forms, to obtain
\[H_\mub^{0,1}=\CC^2\text{ and }H_\mub^{1,0}=\CC.\]
To see this, note as in the proof of Lemma \ref{h20}, that $\mub:\cA^{1,0}_{L}\cong\CC^2 \to \cA^{0,2}_{L}\cong \CC$ is surjective and that $\mub$ is trivial on $\cA^{0,1}_{L}\cong\CC^2$.
Since $H_{\Dol}^{p,q}\cong H^q(H_\mub^{p,*},\delb)$, we have inequalities 
\[2= \dim H_{\mub}^{0,1}\geq \dim H_{\Dol}^{0,1}\text{ and }
1= \dim H_{\mub}^{1,0}\geq \dim H_{\Dol}^{1,0}.
\]
But convergence of the spectral sequence gives 
\[\dim H_{\Dol}^{0,1}+\dim H_{\Dol}^{1,0}\geq b^{k}=3.\]
Combining the two inequalities we deduce that $H_{\Dol}^{1,0}\neq 0$.

The remaining numbers follow from the various dualities.
\end{proof}

Note that we used left-invariant forms to compute dimensions of complete (not necessarily left invariant) harmonic spaces, via the comparison with Dolbeault cohomology and its induced bigrading. Similarly, one can use a comparison with the bigrading of de Rham cohomology. Recall from Remark \ref{Dolrema} that 
there are inequalities $\ell^{p,q}\leq h_{\dR}^{p,q}$, where $h_{\dR}^{*,*}:=\dim E_\infty^{*,*}$ 
is given by the dimensions of the de Rham page of the Fr\"{o}licher-type spectral sequence of \cite{CWDol}. It follows from Nomizu's Theorem \cite{Nomizu} that there are equalities $^{L}h_{\dR}^{*,*}=h_{\dR}^{*,*}$, 
so one may compute these numbers using only left-invariant forms.
Therefore, the inequality $\ell^{p,q}\leq h_{\dR}^{p,q}$ allows one, in good circumstances, to determine the true numbers $\ell^{p,q}$ from the left-invariant numbers.

\subsection{Generalized Hard Lefschetz Theorem} 

As noted after Corollary \ref{PrimDecomp}, the hard Lefschetz Theorem implies that the space of pure $d$-harmonics forms are zero when certain higher Betti numbers vanish. For example, if $M$ is a compact symplectic manifold of dimension $2m$, and any of the Betti numbers $b^3, b^5, \ldots ,b^{2j+1}$ are zero, with $1< 2j+1< 2m$, then for any compatible almost complex structure, we must have $\Omega_\delb^1= 0$.  

Equivalently, on any compact $2m$-manifold, if some $J$ has $\Omega_\delb^1 \neq 0$ and any of $b^3, b^5, \ldots b^{2j+1}$ are zero, with  $1< 2j+1< 2m$, then there is no compatible symplectic form. 

We conclude with a nil-manifold example, which shows that the generalized hard Lefschetz theorem
detects non-existence of invariant almost K\"ahler structures. Consider the real $4$-dimensional nilpotent filiform Lie algebra with basis $X_1 , X_2, X_3, X_4$
and only non-zero brackets
\[
[X_1 ,X_i  ] = X_{i+1} \quad \textrm{for} \quad i=2,3.
\]
One can check that the Betti numbers of the compact quotient filiform manifold $\Gamma  \backslash G$, where $\Gamma$ is a discrete subgroup of the simply connected Lie group $G$, are $b^1= b^2 = b^3 = 2$. 

Consider the non-integrable left invariant almost complex structure given by $JX_1 = X_2$ and $J X_3  =X_4$. 
Letting $A= X_1 -iJ X_1$ and $B= X_3 -iJX_3$, the dual elements $a$ and $b$ are a basis for the invariant $(1,0)$-forms. One can check that the only non-zero components of the exterior differential are
\[
\mub b = \frac{1}{2i} \bar a \bar b, \quad \quad 
\delb b =\frac{1}{2i}  \left( a \bar b - b \bar a \right) - i a \bar a,   \quad \quad
\del b = \frac{1}{2i} ab,
\]
and their conjugates. We'll use hard Lefschetz duality to show there is no left invariant metric making this almost complex manifold into an almost K\"ahler manifold.

First, independent of metric we have $a \in \Hh^{1,0}_{\delb} \cap \Hh^{1,0}_{\mu}$,  so if there were a metric making this almost K\"ahler, then also $\bar a \in \Hh^{0,1}_{\delb} \cap \Hh^{0,1}_{\mu}$, and since $b^1=2$, it follows that for any almost K\"ahler metric $\ell^{0,1} = \ell^{1,0} = 1$. Also, by duality, $\ell^{2,1} = \ell^{1,2} = 1$ as well.

It is straightforward to check that a real basis for the left invariant real $(1,1)$-forms is given by 
\[
\{i a \bar a, i b \bar b , i \left( a \bar b + b \bar a \right) , \left( a \bar b - b \bar a \right) \}.
\]
and that a basis for the $\delb$-closed left invariant real $(1,1)$-forms is given by
\[
\{i a \bar a, i \left( a \bar b + b \bar a \right) \}.
\]
Then any compatible invariant symplectic form can be written as $\omega =  i \alpha 
 a \bar a + i \beta \left( a \bar b + b \bar a \right) $, for some constants $\alpha , \beta$. 
 
 The hard Lefschetz duality theorem implies that
 \[
 L(a) =  i \beta a  b \bar a  \in \Hh^{2,1}_{\delb} \cap   \Hh^{2,1}_\mu.
 \]
 But, $\delb (ab) = \frac i 2 a b \bar a$, so that $L(a) \in \mathrm{Im} (\delb)$ and therefore $\delb^* L(a)  \neq 0$, unless $\beta = 0$. But then $\omega = i \alpha a \bar a$ is degenerate, which is a contradiction.
 
 We note that this manifold does admit a symplectic form. One example is $\omega = x_1 x_4 + x_2 x_3$, which is almost K\"ahler for the metric making $\{X_i\}$ orthonormal, and has $J'X_1 = X_4$ and $J'X_2 = X_3$. This manifold does not admit any integrable almost complex structure, as pointed out to us by Aleksandar Milivojevic. In this case,
  the numbers $\ell^{*,*}$ are given by 
\[
\xymatrix@C=.1pc@R=.1pc{
 &&1\\
 &0&&0\\
 0&&2&&0\\
 &0&&0\\
 &&1
}.
\] 
  
 \subsection{Laplacian identity detects symplectic manifolds with almost complex structures that admit no compatible almost K\"ahler structure}
 
 Consider the $6$-dimensional nilpotent real Lie algebra with basis $\{X_1, \ldots , X_6\}$ and only non-zero brackets given by 
 \[
 [X_1 , X_3 ] = [X_2, X_4] = X_5 \quad \textrm{and} \quad [X_1, X_4] = -[X_2,X_3] = X_6
 \]
 This is considered in \cite{CFGU2}, denoted as $\mathfrak{h}_5$,  and induces a nilmanifold with nilpotent complex structure. Let $G$ be the simply connected group associated to $\mathfrak{h}_5$, and let $\Gamma \backslash G$ be a nilmanifold where $\Gamma$ is a discrete lattice.
 
 Note $\Gamma \backslash G$ is a symplectic manifold, with symplectic form given by 
 \[
 \omega = x_1 x_5 + x_2 x_6 + x_3 x_6,
 \]
 where $x_i$ is dual to $X_i$. 
 
 Consider the almost complex structure given by 
  \[
 JX_1  = X_2 \quad J X_3 = -X_4 \quad JX_5 = -X_6.
 \]
 Define generators $A := X_5 - iJX_5 = X_5 +iX_6$, $B := X_1 - iJX_1= X_1-iX_2$, and $C := X_3 - iJX_3 = X_3+iX_4$
 so that
 \[
 [B,C] = [X_1-iX_2,  X_3+iX_4] = 2(X_5+iX_6) = 2A.
 \]
 Letting $a,b,c$ denote the duals of $A,B,C$ respectively, we have \[da= -\frac 1 2 bc= \del a.\]
  Therefore, $a \in \Ke \left( \Delta_\delb + \Delta_\mu \right)$ but  $ a \notin \Ke \left( \Delta_\del + \Delta_\mub \right)$. Hence,
 \[
  \Delta_\delb + \Delta_\mu  \neq  \Delta_\del + \Delta_\mub,
 \]
 so by Theorem \ref{AKDeltas}, we can conclude that this almost complex structure does not admit a compatible symplectic structure. Of course, $db = dc =0$ as well, so that $J$ is integrable. So, one could arrive at this same conclusion since $\Gamma \backslash G$ is not K\"ahler, since for example $\ell^{1,0} \geq 3  > \frac{b^1}{2} = 2$.
 
 But, the example shows even more: if $(M,J')$ is any almost complex manifold (integrable or not), then the Cartesian product 
 $\Gamma \backslash G \times M$ with product almost complex structure $J \times J'$, does not admit any metric making the product an almost K\"ahler manifold. Indeed, the pullback of the form $a$ to the product still exhibits that $ \Delta_\delb + \Delta_\mub  \neq  \Delta_\del + \Delta_\mu$ on $\Gamma \backslash G \times M$. 
 
In particular, if $M$ is symplectic then $\Gamma \backslash G \times M$ always admits a symplectic form, but never a symplectic form compatible with this almost complex structure.

\bibliographystyle{alpha}

\bibliography{../biblio}

\begin{thebibliography}{CFGU00}

\bibitem[AGF81]{AF}
L.~Alvarez-Gaum\'e and D.~Z. Freedman.
\newblock Geometrical structure and ultraviolet finiteness in the
  supersymmetric {$\sigma $}-model.
\newblock {\em Comm. Math. Phys.}, 80(3):443--451, 1981.

\bibitem[B{\"{a}}r97]{Bar}
Christian B{\"{a}}r.
\newblock On nodal sets for {D}irac and {L}aplace operators.
\newblock {\em Comm. Math. Phys.}, 188(3):709--721, 1997.

\bibitem[BG88]{BG}
C.~Benson and C.~S. Gordon.
\newblock K\"ahler and symplectic structures on nilmanifolds.
\newblock {\em Topology}, 27(4):513--518, 1988.

\bibitem[BMn16]{BM}
G.~Bazzoni and V.~Mu\~noz.
\newblock Manifolds which are complex and symplectic but not {K}\"ahler.
\newblock In {\em Essays in mathematics and its applications}, pages 49--69.
  Springer, [Cham], 2016.

\bibitem[Bry88]{Brylinski}
Jean-Luc Brylinski.
\newblock A differential complex for {P}oisson manifolds.
\newblock {\em J. Differential Geom.}, 28(1):93--114, 1988.

\bibitem[CFGU00]{CFGU2}
L.~A. Cordero, M.~Fern\'andez, A.~Gray, and L.~Ugarte.
\newblock Compact nilmanifolds with nilpotent complex structures: {D}olbeault
  cohomology.
\newblock {\em Trans. Amer. Math. Soc.}, 352(12):5405--5433, 2000.

\bibitem[Che71]{Chen}
K.~Chen.
\newblock Algebras of iterated path integrals and fundamental groups.
\newblock {\em Trans. Amer. Math. Soc.}, 156:359--379, 1971.

\bibitem[Che72]{Chen2}
K.~Chen.
\newblock Differential forms and homotopy groups.
\newblock {\em J. Differential Geometry}, 6:231--246, 1971/72.

\bibitem[CW18]{CWDol}
J.~Cirici and S.O. Wilson.
\newblock Dolbeault cohomology for almost complex manifolds.
\newblock {\em Preprint arxiv:1809.1416}, 2018.

\bibitem[dBT01]{BaTo}
P.~de~Bartolomeis and A.~Tomassini.
\newblock On formality of some symplectic manifolds.
\newblock {\em Internat. Math. Res. Notices}, (24):1287--1314, 2001.

\bibitem[DLZ10]{DrLiZh}
T.~Draghici, T.-J. Li, and W.~Zhang.
\newblock Symplectic forms and cohomology decomposition of almost complex
  four-manifolds.
\newblock {\em Int. Math. Res. Not. IMRN}, (1):1--17, 2010.

\bibitem[Don90]{Donaldson}
S.~K. Donaldson.
\newblock Yang-{M}ills invariants of four-manifolds.
\newblock In {\em Geometry of low-dimensional manifolds, 1 ({D}urham, 1989)},
  volume 150 of {\em London Math. Soc. Lecture Note Ser.}, pages 5--40.
  Cambridge Univ. Press, Cambridge, 1990.

\bibitem[FGR98]{FGR}
J.~Fr\"ohlich, O.~Grandjean, and A.~Recknagel.
\newblock Supersymmetric quantum theory and differential geometry.
\newblock {\em Comm. Math. Phys.}, 193(3):527--594, 1998.

\bibitem[FT10]{FTAC}
A.~Fino and A.~Tomassini.
\newblock On some cohomological properties of almost complex manifolds.
\newblock {\em J. Geom. Anal.}, 20(1):107--131, 2010.

\bibitem[GH94]{GrHa}
P.~Griffiths and J.~Harris.
\newblock {\em Principles of algebraic geometry}.
\newblock Wiley Classics Library. John Wiley \& Sons, Inc., New York, 1994.
\newblock Reprint of the 1978 original.

\bibitem[Gom95]{Gompfreal}
R.~E. Gompf.
\newblock A new construction of symplectic manifolds.
\newblock {\em Annals of Math.}, 142:527--595, 1995.

\bibitem[Gom01]{Gompf}
R.~E. Gompf.
\newblock The topology of symplectic manifolds.
\newblock {\em Turkish J. Math.}, 25(1):43--59, 2001.

\bibitem[Hir54]{Hirzebruch}
F.~Hirzebruch.
\newblock Some problems on differentiable and complex manifolds.
\newblock {\em Ann. of Math. (2)}, 60:213--236, 1954.

\bibitem[HKLR87]{HKLR}
N.~J. Hitchin, A.~Karlhede, U.~Lindstr\"om, and M.~Ro\v{c}ek.
\newblock Hyper-{K}\"ahler metrics and supersymmetry.
\newblock {\em Comm. Math. Phys.}, 108(4):535--589, 1987.

\bibitem[Huy05]{Huy}
D.~Huybrechts.
\newblock {\em Complex geometry}.
\newblock Universitext. Springer-Verlag, Berlin, 2005.
\newblock An introduction.

\bibitem[HZ20]{HZ}
T.~Holt and W.~Zhang.
\newblock Harmonic forms on the {K}odaira-{T}hurston manifold.
\newblock {\em Preprint arxiv:2001.10962}, 2020.

\bibitem[Kod64]{Kod}
K.~Kodaira.
\newblock On the structure of compact complex analytic surfaces. {I}.
\newblock {\em Amer. J. Math.}, 86:751--798, 1964.

\bibitem[Kot97]{KotschickBourbaki}
D.~Kotschick.
\newblock The {S}eiberg-{W}itten invariants of symplectic four-manifolds (after
  {C}. {H}. {T}aubes).
\newblock {\em Ast\'erisque}, (241):Exp.\ No.\ 812, 4, 195--220, 1997.
\newblock S\'eminaire Bourbaki, Vol. 1995/96.

\bibitem[LZ09]{LZ}
T.-J. Li and W.~Zhang.
\newblock Comparing tamed and compatible symplectic cones and cohomological
  properties of almost complex manifolds.
\newblock {\em Comm. Anal. Geom.}, 17(4):651--683, 2009.

\bibitem[Mat95]{Mathieu}
O.~Mathieu.
\newblock Harmonic cohomology classes of symplectic manifolds.
\newblock {\em Comment. Math. Helv.}, 70(1):1--9, 1995.

\bibitem[Nom54]{Nomizu}
K.~Nomizu.
\newblock On the cohomology of compact homogeneous spaces of nilpotent {L}ie
  groups.
\newblock {\em Ann. of Math. (2)}, 59:531--538, 1954.

\bibitem[Tau94]{Taubes1}
C.~H. Taubes.
\newblock The {S}eiberg-{W}itten invariants and symplectic forms.
\newblock {\em Math. Res. Lett.}, 1(6):809--822, 1994.

\bibitem[Tau95]{Taubes2}
C.H. Taubes.
\newblock More constraints on symplectic forms from {S}eiberg-{W}itten
  invariants.
\newblock {\em Math. Res. Lett.}, 2(1):9--13, 1995.

\bibitem[Thu76]{Thur}
W.~P. Thurston.
\newblock Some simple examples of symplectic manifolds.
\newblock {\em Proc. Amer. Math. Soc.}, 55(2):467--468, 1976.

\bibitem[TT19]{TaTopreprint}
N.~Tardini and A.~Tomassini.
\newblock Differential operators on almost-{H}ermitian manifolds and harmonic
  forms.
\newblock {\em to appear in Complex Manifolds}, preprint arXiv:1909.06569,
  2019.

\bibitem[TY12a]{TY1}
L.-S. Tseng and S.-T. Yau.
\newblock Cohomology and {H}odge theory on symplectic manifolds: {I}.
\newblock {\em J. Differential Geom.}, 91(3):383--416, 2012.

\bibitem[TY12b]{TY2}
L.-S. Tseng and S.-T. Yau.
\newblock Cohomology and {H}odge theory on symplectic manifolds: {II}.
\newblock {\em J. Differential Geom.}, 91(3):417--443, 2012.

\bibitem[Ver11]{Verbitsky}
M.~Verbitsky.
\newblock Hodge theory on nearly {K}\"ahler manifolds.
\newblock {\em Geom. Topol.}, 15(4):2111--2133, 2011.

\bibitem[Voi07]{Voisin}
C.~Voisin.
\newblock {\em Hodge theory and complex algebraic geometry. {I}}, volume~76 of
  {\em Cambridge Studies in Advanced Mathematics}.
\newblock Cambridge University Press, Cambridge, english edition, 2007.
\newblock Translated from the French by Leila Schneps.

\bibitem[Wei58]{Weil}
A.~Weil.
\newblock {\em Introduction \`a l'\'etude des vari\'et\'es k\"ahl\'eriennes}.
\newblock Publications de l'Institut de Math\'ematique de l'Universit\'e de
  Nancago, VI. Actualit\'es Sci. Ind. no. 1267. Hermann, Paris, 1958.

\bibitem[Yan96]{Yan}
D.~Yan.
\newblock Hodge structure on symplectic manifolds.
\newblock {\em Adv. Math.}, 120(1):143--154, 1996.

\bibitem[Zum79]{Zumino}
B.~Zumino.
\newblock {Supersymmetry and {K}\"ahler Manifolds}.
\newblock {\em Phys. Lett.}, 87B:203, 1979.

\end{thebibliography}

\end{document}